\documentclass[12pt]{amsart}
\usepackage{amsmath}
\usepackage{amscd}
\usepackage{amssymb}
\usepackage{amsfonts}
\usepackage{amsthm}
\usepackage{bbm}
\usepackage{cancel}
\usepackage{color}
\usepackage{eucal}
\usepackage{enumerate,yfonts}
\usepackage{enumitem}
\usepackage{pdfsync}
\usepackage[all,cmtip]{xy}
\usepackage{graphicx}
\usepackage{graphics}
\usepackage{hyperref}
\usepackage{latexsym}
\usepackage{mathrsfs}
\usepackage{placeins}
\usepackage{pstricks}
\usepackage{bookmark}
\usepackage{mathtools}
\usepackage{stmaryrd}
\usepackage{url}
\usepackage{mathabx}
\usepackage{float}

\newtheorem{thm}{Theorem}

\newtheorem{proposition}[thm]{Proposition}

\newtheorem{conjecture}[thm]{Conjecture}
\newtheorem{thm-dfn}[thm]{Theorem-Definition}

\newenvironment{rouge}
{\relax\color{red}}
{\hspace*{.3ex}\relax}

\newenvironment{bluet}
{\relax\color{blue}}
{\hspace*{.3ex}\relax}

\newcommand{\br}{\begin{rouge}}
\newcommand{\er}{\end{rouge}}

\newcommand{\bb}{\begin{bluet}}
\newcommand{\eb}{\end{bluet}}

\newcommand{\nc}{\newcommand}

\newcommand{\cB}{{\mathcal B}}

\newcommand{\cO}{{\mathcal O}}

\newcommand{\cH}{{\mathcal H}}

\newcommand{\cL}{{\mathcal L}}
\newcommand{\cE}{{\mathcal E}}

\newcommand{\bC}{{\mathbb C}}

\newcommand{\Lg}{{\mathfrak{g}}}

\newcommand{\fb}{{\mathfrak{b}}}

\newcommand{\ft}{{\mathfrak{t}}}
\newcommand{\fg}{{\mathfrak g}}

\newcommand{\fn}{{\mathfrak{n}}}

\nc{\ot}{\otimes}
\nc{\on}{\operatorname}

\nc{\Lie}{{\operatorname{Lie}}}
\nc{\oh}{{\operatorname{H}}}
\nc{\gr}{{\operatorname{gr}}}
\nc{\rk}{{\operatorname{rank}}}
\nc{\codim}{{\operatorname{codim}}}
\nc{\img}{{\operatorname{Im}}}
\nc{\IC}{{\operatorname{IC}}}

\nc{\bI}{{\mathbf 1}}

\nc{\lp}{{\left(}}
\nc{\rp}{{\right)}}

\newcommand{\beqn}{\begin{equation*}}
\newcommand{\eeqn}{\end{equation*}}

\newcommand{\beq}{\begin{equation}}
\newcommand{\eeq}{\end{equation}}

\newcommand{\bega}{\begin{gathered}}
\newcommand{\eega}{\end{gathered}}

\newcommand{\bern}{\begin{eqnarray*}}
\newcommand{\eern}{\end{eqnarray*}}
\newcommand{\ber}{\begin{eqnarray}}
\newcommand{\eer}{\end{eqnarray}}

\nc{\diag}{{\on{diag}}}

\nc{\Ts}{{T_{\mathbf{s}}}}
\nc{\Tds}{{\check T_{\mathbf{s}}}}

\begin{document}

\title[Hessenberg]{A note on Hessenberg varieties}

\author{Kari Vilonen}
\address{School of Mathematics and Statistics, University of Melbourne, VIC 3010, Australia, also Department of Mathematics and Statistics, University of Helsinki, Helsinki, Finland}
\email{kari.vilonen@unimelb.edu.au, kari.vilonen@helsinki.fi}
\thanks{KV was supported in part by the ARC grants DP150103525 and DP180101445  and the Academy of Finland}

\author{Ting Xue}
\address{School of Mathematics and Statistics, University of Melbourne, VIC 3010, Australia, also Department of Mathematics and Statistics, University of Helsinki, Helsinki, Finland} 
\email{ting.xue@unimelb.edu.au}
\thanks{TX was supported in part by the ARC grant DP150103525.}

\maketitle

We give a short proof based on Lusztig's generalized Springer correspondence of some results of~\cite{BrCh,BaCr,P}.

Let $\fg$ denote a complex semisimple Lie algebra and let us write $G=\on{Aut}(\fg)^0$ for the identity component of its group of automorphisms. Note that the group $G$ is adjoint. Let us write $\cB$ for the flag variety of $\fg$ and let us fix a Borel subalgebra $\fb\subset\fg$. Let $B\subset G$ be the corresponding Borel subgroup. 
Consider a Hessenberg subspace $V\subset \fg$, i.e., a $B$-invariant subspace of $\fg$. It gives rise to a $G$-equivariant bundle
\beq
\cH := G\times ^B V \to \cB\,.
\eeq
We have a $G$-equivariant projective morphism
\beq
\pi : \cH \to \fg\,,\ (g,x)\mapsto\on{Ad}(g)x
\eeq
the ``universal" family of Hessenberg varieties. 
We can also consider a perpendicular Hessenberg subspace $V^\perp$ and we similarly get
\beq
\check\pi : \check\cH:=G\times^B V^\perp \to \fg.
\eeq
We now assume that $\fb \subset V$ so that $V^\perp \subset \fn=[\fb,\fb]$. 
By the decomposition theorem~\cite{BBD} we have
\beq
\label{1}
R\pi_* \bC_{\cH} \ = \ \oplus \on{IC}(\fg^{rs}, \cL_i)[-] \oplus \text{terms with smaller support}\,;
\eeq
here the $\cL_i$ are irreducible $G$-equivariant local systems on the regular semisimple locus $\fg^{rs}$ and $[-]$ denotes cohomological shifts. 
We also have 
\beq
\label{2}
R\check\pi_* \bC_{\check\cH} \ = \ \oplus \on{IC}(\cO, \cE)[-]\,;
\eeq
here the $\cO$ are nilpotent orbits of $G$ and the $\cE$ are $G$-equivariant irreducible local systems on the $\cO$. 
By considering the Fourier transform we conclude that all the $\on{IC}$'s appearing in~\eqref{1} are Fourier transforms of the $\on{IC}$'s appearing in~\eqref{2}. 

We have the following
\begin{conjecture} [Patrick Brosnan \cite{X}]
All the $\on{IC}(\cO, \cE)$ in~\eqref{2} appear in the Springer correspondence. In particular, all the terms in~\eqref{1} have full support. 
\end{conjecture}
This conjecture has been proved by Martha Precup and Eric Sommers (for type $G_2$ see also~\cite{X}). It is not difficult to prove the following weaker statement which suffices for some applications. 
\begin{proposition}
We have
$$
(R\pi_* \bC_\cH)|_{\fg^{reg}}\ = \ \oplus \on{IC}(\fg^{rs}, \cL_i)|_{\fg^{reg}}[-]\,.
$$
\end{proposition}
\begin{proof}
The terms in~\eqref{1} which do not appear in the Springer correspondence come by induction from the (nontrivial) cuspidals in Lusztig's generalized Springer correspondence~\cite{L}. Recall that we are in the context of the adjoint group. One can now verify, case-by-case using~\cite{L}, that  the supports of  such terms do not meet the regular locus $\Lg^{reg}$. 
\end{proof}

The $\cL_i$ come from representations of the Weyl group $W$ and we will now denote them as such by $\rho_i$. 
 Let us consider 
$\on{IC}(\fg^{rs}, \rho_i)$ restricted to the regular locus. It is a direct summand in $\tilde p_*\bC_{\tilde \fg^{reg}}$ where $\tilde p:\tilde \fg^{reg}\to \fg^{reg}$ is the  Grothendieck simultaneous resolution restricted to the regular locus. Let $\ft$ be a Cartan subspace of $\fg$. Then we have the following Cartesian square
\beq
\begin{CD}
\tilde \fg^{reg}@>{\tilde p}>> \fg^{reg}
\\
@VVV @VVfV
\\
\ft @>>p> \ft/W
\end{CD}
\eeq
and so we have $\tilde p_*\bC_{\tilde \fg^{reg}}= f^*p_*\bC_{\ft}$. Thus, as $f$ is smooth, it suffices to analyze $p_*\bC_{\ft}$.

Let us now consider a stratum $\cO_s$ in $\ft/W$ associated to a semisimple element $s\in\ft$. Note that the  fundamental group of $\cO_s$ is the braid group of $^sW= N_G(L)/L$ where $L=Z_G(s)$. We also write $W_s=\on{Stab}_W(s)$. By observing that $N_G(L)/L = (N _G(L)\cap N _G(T))/N_L(T)$ we have an exact sequence 
\beq
\label{e}
1 \to W_s= N_L(T)/T \to (N _G(L)\cap N _G(T))/T \to {}^sW= N_G(L)/L \to 1\,.
\eeq
The fiber of $p$ over this stratum is $W/W_s$. This implies:
\beq
\on{IC}(\ft/W, \rho_i)|_{\cO_s}\ = \ \rho_i^{W_s}\,.
\eeq
The righthand side can be viewed as a representation of ${}^sW$ by~\eqref{e} and hence as a local system on $\cO_s$. Note that, of course, all the $\on{IC}$'s are just sheaves so in the whole analysis talking about $\on{IC}$'s is not really necessary.

The considerations above allow us to conclude that 
\beq
\oh^*(\cH_{x}) \ = \ \oh^*(\cH_{y})^{W_s}
\eeq
where $x$ is a regular element with semisimple part $s$ and $y$ is a regular semisimple element. In particular, this covers some of the results of~\cite{BrCh,BaCr} as well as the palindromicity of the cohomology of regular Hessenberg varieties~\cite{P}.

\end{document}